\documentclass[12pt,a4paper]{article}
 
\usepackage[a4paper,portrait]{geometry}
\usepackage{amsmath,amsfonts,latexsym,amssymb,amsthm}
\usepackage{graphicx}
\usepackage{epsfig}

\newcommand{\DOIFPDF}[2]{\ifx\pdfoutput\undefined #2\else#1\fi}

\newcommand{\EM}{\ensuremath}

\newcommand{\PDFABLE}[2]{%
\newif\ifpdf%
\ifx\pdfoutput\undefined\pdffalse%
\else\pdftrue\pdfoutput=1\pdfcompresslevel=9\fi%
\ifpdf%
 \usepackage#1
 \usepackage[pdftex,%
             a4paper,%
             colorlinks,%
             citecolor=blue,%
             pagebackref,%
             plainpages=false]{hyperref}%
\else%
 \usepackage#2
 \usepackage{url}
\fi}

\makeatletter
\newcommand{\@THMSTYLES}{%
  \newtheoremstyle{bodyrm}
  {3pt}
  {3pt}
  {}
  {}
  {\bfseries\sffamily}
  {.}
  { }
  {}
  \newtheoremstyle{bodyit}
  {3pt}
  {3pt}
  {\itshape}
  {}
  {\bfseries\sffamily}
  {.}
  { }
  {}
}
\newcommand{\THMEN}{%
  \@THMSTYLES
  \theoremstyle{bodyit}
  \newtheorem{thm}{Theorem}[section]%
  \newtheorem{cor}[thm]{Corollary}%
  \newtheorem{prop}[thm]{Proposition}%
  \newtheorem{lem}[thm]{Lemma}%
  \theoremstyle{bodyrm}%
  \newtheorem{defi}[thm]{Definition}%
  \newtheorem{xpl}[thm]{Example}%
  \newtheorem{exo}[thm]{Exercise}%
  \newtheorem{hyp}[thm]{Hypothesis}%
  \newtheorem{eur}[thm]{Heuristics}%
  \newtheorem{pro}[thm]{Problem}%
  \newtheorem{rem}[thm]{Remark}%
  \newtheorem{prp}[thm]{Property}%
}
\newcommand{\THMFR}{%
  \@THMSTYLES
  \theoremstyle{bodyit}
  \newtheorem{thm}{Théorème}[section]%
  \newtheorem{prop}[thm]{Proposition}%
  \theoremstyle{bodyrm}%
}
\makeatother

\makeatletter
\newcommand{\SMALLSECS}{%
 \renewcommand{\section}{\@startsection%
  {section}
  {1}
  {0em}
  {\baselineskip}
  {0.5\baselineskip}
  {\normalfont\large\bfseries}}
 \renewcommand{\subsection}{\@startsection%
  {subsection}
  {2}
  {0em}
  {\baselineskip}
  {0.25\baselineskip}
  {\normalfont\bfseries}}
}
\makeatother

\makeatletter
\providecommand{\timenow}{\@tempcnta\time
\@tempcntb\@tempcnta
\divide\@tempcntb60
\ifnum10>\@tempcntb0\fi\number\@tempcntb
\multiply\@tempcntb60
\advance\@tempcnta-\@tempcntb
:\ifnum10>\@tempcnta0\fi\number\@tempcnta}
\makeatother

\makeatletter
\newcommand{\versiondetravail}{%
 \renewcommand{\@evenfoot}{%
 \hfil{\tiny\texttt{%
   Version préliminaire, compilée le \today{} à \timenow.}\hfill}}%
 \renewcommand{\@oddfoot}{\@evenfoot}%
}
\makeatother

\newcommand{\bE}{\EM{\mathbf{E}}}

\newcommand{\p}[4]{{#3}\!\left#1{#4}\right#2}

\newcommand{\covf}[1]{\mathbf{Cov}_{#1}}
\newcommand{\cov}[3]{\p(){\covf{#1}}{#2,#3}}

\newcommand{\Det}[1]{\mathrm{Det}\,}

\newcommand{\D}{\mathbf{D}}

\renewcommand{\leq}{\leqslant}
\renewcommand{\geq}{\geqslant}

\SMALLSECS
\THMEN
\title{Non parametric estimation of the structural expectation of a stochastic increasing function}
\author{J-F.~\textsc{Dupuy} \& J-M.~\textsc{Loubes} \& E.~\textsc{Maza}}
\date{}

\newcommand{\mykeywords}{
Functional data analysis; Non parametric warping model; Structural expectation; Curve registration;
.}

\newcommand{\mysubjclass}{
62G05, 62G20
.}

\newcommand{\rr}{\ensuremath{\mathrm{r}}}
\newcommand{\Prob}{{\mathbf{P}}}

\newtheorem{theorem}{Theorem}[section]
\newtheorem{lemma}[theorem]{Lemma}
\newtheorem{proposition}[theorem]{Proposition}
\newtheorem{corollary}[theorem]{Corollary}
\begin{document}

\maketitle
{\small
\begin{abstract}
This article introduces a non parametric warping model for functional data. When the outcome of an experiment is a sample of curves, data can be seen as realizations of a stochastic process, which takes into account the small variations between the different observed curves. The aim of this work is to define a mean pattern which represents the main behaviour of the set of all the realizations. So we define the structural expectation of the underlying stochastic function. Then we provide   empirical estimators of this structural expectation and of each individual warping function. Consistency and asymptotic normality for such estimators are proved.\\
\end{abstract}
{ \noindent
 \textbf{Subject Class. MSC-2000}: \mysubjclass}
}
\section{Statistical model for functional phase variations}\label{s_model}
Functional data analysis deals with the analysis of experiments where  one or several quantities are recorded during a time period for different individuals, resulting in a sample of observed curves. One of the main difficultis is given by the fact that curves usually not only present amplitude variability (a variation in the $y$-axis) but also time or phase variability (a variation in the $x$-axis). Hence the classical cross-sectional mean does not make sense and the definition of an appropriate population mean is even not obvious. Giving a sense to the common behaviour of a sample of curves, and finding a meaningful mean curve in this context is thus an important issue, called curve registration, or time warping problem, which  first appeared in the engineering literature in \cite{Sakoe78}. \vskip .1in
Several methods have been proposed over the years to estimate the mean pattern of a sample of curves. A popular method consists in i) first aligning the curves to a given template by warping the time axis, ii) then taking the mean of all the dewarped curves. Such methods are increasingly common in statistics, see \cite{Ramsay02} for a review.  A landmark registration methodology is proposed by \cite{Kneip92} and further developed by \cite{jbigot}.  A non parametric method is investigated in \cite{Ramsay98} and in \cite{Kneip00}, using  local regressions.  Dynamic time warping methodology is developed by \cite{Wang99}.  An alternative approach is provided  in \cite{Gamboa05}, where semi-parametric estimation of shifted curves is studied. But these methods imply choosing a starting curve as a fixed point for the alignment process. This initial choice may either bias the estimation procedure, or imply strong and restrictive identifiability conditions.\vskip .1in
In this work, we consider a second point of view. We define an archetype representing the common behaviour of the sample curves directly from the data, without stressing a particular curve. Such a method has the advantage of not assuming any technical restrictions on the data and so, enables to handle a large variety of cases. However, the registration procedure and the common pattern have to be clearly defined. \\
\indent We observe $i=1,\dots,m$ curves $f_i:[a,b] \to \mathbb{R}$ at equidistant discrete times $t_{ij}\in [a,b],\: j=1,\dots,n$. So the data can be written as
\begin{equation}\label{f_dataset}
Y_{ij}=f_i\left(t_{ij}\right),\:i=1,\dots,m,\:j=1,\dots,n.
\end{equation}
So, the registration problem aims at  finding  a mean pattern $f$ and warping functions $h_i$ which align the observed curves, i.e such that $\forall i=1,\dots,m,\: f=f_i \circ h_i$. Hence each curve is obtained by warping the original curve $f$ using the  warping functions $h_i$.\vskip .1in 
Defining the registration operator is a difficult task. In this paper we propose a random warping procedure which takes into account the variability of the deformation as a random effect. So, we assume that there exists a random process $H$ such that the data are i.i.d realizations of this process.
Let $H$ be this warping stochastic process  defined as
\begin{equation}\label{d_h}
\begin{array}{rrcl}
H:&\Omega&\rightarrow&\mathcal{C}\left([a,b]\right)\\
&w&\mapsto&H(w,\cdot),
\end{array}
\end{equation}
where $\left(\Omega,\mathcal{A},\Prob\right)$ is an unknown probability space, and $\left(\mathcal{C}\left([a,b]\right),\left\|\cdot\right\|_\infty,\mathcal{B}\right)$ is the set of all real continuous functions defined on the interval $[a,b]$, equipped with the uniform norm and with its Borel algebra. Consider $h_1,\dots,h_m$ i.i.d realizations of the process $H(t)$. $h_i$  warps a mean pattern $f$ onto the $i$-th observation curve $f_i$. Hence, model \eqref{f_dataset} can be modeled by
\begin{equation}\label{m_model}
Y_{ij}=f_i(t_{ij})=f\circ h_i^{-1}(t_{ij}).
\end{equation}
For sake of simplicity, we will write $f_i(t)=f\circ h_i^{-1}(t)$, for all $i\in{\{1,\dots,m\}}$. We point out that $h_i^{-1}$ is well defined since the warping processes are assumed to be continuous increasing functions.\vskip .1in
Under general assumptions \eqref{d_h}, model \eqref{m_model} is not identifiable. More precisely, the unknown function $f$ and the unknown warping process $H$ can not be estimated. Indeed, if $\tilde{h}:[a,b]\to[a,b]$ is an increasing continuous function, with $\tilde{h}(a)=a$ and $\tilde{h}(b)=b$, then, for all $i\in{\{1,\dots,m\}}$ and all $j\in{\{1,\dots,n\}}$, we have that
$Y_{ij}=f\circ\tilde{h}^{-1}\circ\tilde{h}\circ h_i^{-1}(\cdot,t_{ij}).$
Hence, the function $f\circ\tilde{h}^{-1}$, associated with the warping process $H\circ\tilde{h}^{-1}$, is also a solution of model \eqref{m_model}.\vskip .1in
 The aim of this paper is to build a new kind of pattern which represents the common functional feature of the data but still that can be estimated together with the warping procedure.  For this, let $\phi(.)$ be the expectation of the warping process and  define the {\sl structural expectation} ${f}_{ES}$ as 
$${f}_{ES}:=f\circ\phi^{-1}.$$
The structural expectation is obviously not the function $f$, but the function $f$ composed with $\phi^{-1}$, the inverse of the expectation of $H$. Hence it can  be seen as the mean warping of the function $f$ by the stochastic process $H$. In this article, we aim at studying the properties of the structural expectation, and finding an estimator of the structural expectation and of the warping paths $h_i$.\vskip .1in
The article is organized as follows. In Section~\ref{s_model}, we introduce a warping functional model with a stochastic phase variation, and we introduce the structural expectation.  In Section~\ref{s_estimators}, we define empirical estimators of the structural expectation and of the individual warping functions. Asymptotic properties of these estimators are investigated. Proofs are postponed to Section~\ref{s_proofs}.  Section~\ref{s:cmieux} investigates some extensions of the proposed methodology, in particular to the case of noisy non increasing functions. The results of a simulation study are reported in Section~\ref{s_data}. There, we also apply the proposed estimators to a real data set.
\section{Theoretical study of a particular case: warping  of strictly increasing functions}\label{s_estimators}
First, consider the case where $f$ is a strictly increasing. Hence the inverse function $f^{-1}$ exists and is also strictly increasing. Moreover, a phase warping of function $f$ (i.e. on $x$-axis) corresponds to an amplitude warping of function $f^{-1}$ (i.e. on $y$-axis). We propose estimators of  both the inverse of the structural expectation $f_{ES}^{-1}=\phi\circ f^{-1}$, and also of each individual warping function $\phi\circ h_i^{-1}$, for all $i\in{\{1,\dots,m\}}$, and finally an of the structural expectation $f_{ES}=f\circ\phi^{-1}$.\vskip .1in
Note that all the asymptotic results are taken with respect to $m$ and $n$, so we recall that $u(m,n)\xrightarrow[m,n\to\infty]{}c$
if, and only if, we have $$\forall\epsilon>0,\:\exists(m_0,n_0)\in\mathbb{N}^2,\:\forall(m,n)\in\mathbb{N}^2,m>m_0\mbox{ and }n>n_0\Rightarrow\left|u(m,n)-c\right|<\epsilon.$$
We will assume the following conditions on the warping process in order to define a good registration procedure.\\
 The warping process does not change the timeline (not time inversion) and leaves fixed the two extreme points, so for almost all $w\in\Omega$,  assume that
\begin{itemize}
\item[i)] $H(w,\cdot)$ is an increasing function,
\item[ii)] $H(w,a)=a$ and $H(w,b)=b$.
\end{itemize}
The following proposition introduces respectively the expectation, the second order moment and the covariance function of $H$.
\begin{prop}
Under assumption \eqref{d_h}, the expectation $\phi(\cdot)$, the second order moment $\gamma(\cdot)$ and the covariance function $\mathrm{r}(\cdot,\cdot)$ of the stochastic process $H$ are well defined. $\phi$ and $\gamma$  are also continuous increasing functions. Moreover, we have $\phi(a)=a$, $\phi(b)=b$, $\gamma(a)=a^2$ and $\gamma(b)=b^2$. As a consequence, we have that ${\rm var} H(\cdot,a)={\rm var} H(\cdot,b)=0$.
\end{prop}
\begin{proof}
The process $H$ is bounded and increasing. As a consequence, $\phi$ and $\gamma$ exist. Moreover, $H$ is a continuous increasing process, which leads to continuous and increasing first and second order moments.
\end{proof}
In order to prove asymptotic results, the following technical assumptions on the warping process $H$ and on the function $f$ are needed:
\begin{description}
\item[Assumptions] 
\end{description}
\begin{enumerate}
\item There exists a constant $C_1>0$ such that, for all $(s,t)\in[f(a),f(b)]^2$, we have
\begin{eqnarray}\label{h_1}
\bE\left|H(s)-\bE H(s)-\left(H(t)-\bE H(t)\right)\right|^2\leq C_1|s-t|^2.
\end{eqnarray}
\item There exists a constant $C_2>0$ such that, for all $(s,t)\in[f(a),f(b)]^2$, we have
\begin{eqnarray}\label{h_2}
\left|f^{-1}(s)-f^{-1}(t)\right|^2\leq C_2|s-t|^2.
\end{eqnarray}
\item There exists a constant $C_3>0$ such that, for all $\omega\in\Omega$, for all $(s,t)\in[a,b]^2$, we have
\begin{eqnarray}\label{h_3}
\left|H^{-1}(\omega,s)-H^{-1}(\omega,t)\right|^2\leq C_3|s-t|^2.
\end{eqnarray}
\end{enumerate}
\subsection{Estimator of the structural expectation $f_{ES}$}
Since $f_i=f \circ h_i^{-1}$ we have $f_i^{-1}=h_i \circ f^{-1}$. Hence ${\bf E}(f_i)=({\bf E} (H)) \circ f^{-1}$. Hence it seems natural to consider the mean of functions $\left(f_i^{-1}\right)_{i\in{\{1,\dots,m\}}}$ in order to estimate the inverse of the structural expectation. For all $y\in[f(a),f(b)]$, and for all $i\in{\{1,\dots,m\}}$, define
\begin{equation}\label{approxt}
j_i(y)=\arg\min_{j\in{\{1,\dots,n\}}}\left|Y_{ij}-y\right| \quad \mbox{ and } \quad T_i(y):=t_{ij_i(y)}.
\end{equation}
Then, the empirical estimator of the inverse of the structural expectation is defined by
\begin{equation}\label{e_fm1}
\widehat{{f_{ES}^{-1}}}(y)=\frac{1}{m}\sum_{i=1}^{m}T_i(y).
\end{equation}
The following theorem provides consistency and asymptotic normality of estimator \eqref{e_fm1}.
\begin{theorem}[Consistency of the inverse of the structural expectation]\label{t_fm1}
Under Assumption \eqref{d_h},  
$$\left\Vert\widehat{{f_{ES}^{-1}}}-{f_{ES}^{-1}}\right\Vert_\infty\xrightarrow[m,n\to\infty]{as}0.$$
Moreover, let $n=m^{\frac{1}{2}+\alpha}$ with $\alpha>0$, and assume Conditions \eqref{h_1} and \eqref{h_2}. Then,  
$$\sqrt{m}\left(\widehat{{f_{ES}^{-1}}}-{f_{ES}^{-1}}\right)\xrightarrow[m\to\infty]{\D}G,$$
where $G$ is a centered Gaussian process  with  covariance given by: for all $(s,t)\in[f(a),f(b)]^2$, 
$${\rm cov}(G(s),G(t))=\rr\left(f^{-1}(s),f^{-1}(t)\right).$$
\end{theorem}
From (\ref{e_fm1}) and (\ref{approxt}), $\widehat{{f_{ES}^{-1}}}$ is an increasing step function with jumps occuring at say, $K(m,n)$ points $v_1,\ldots,v_{K(m,n)}$ in $[f(a),f(b)]$, such that $f(a)= v_0<v_1< \ldots< v_{K(m,n)}< v_{K(m,n)+1}=f(b)$.
Hence for all $y\in[f(a),f(b)]\backslash (v_k)_{k\in \mathcal K}$ $(\mathcal K=\{0,\ldots,K(m,n)+1\})$, $\widehat{{f_{ES}^{-1}}}(y)$ can be expressed as $$\widehat{{f_{ES}^{-1}}}(y)=\sum_{k=0}^{K(m,n)}u_k {\bf 1}_{(v_k,v_{k+1})}(y)$$ with $a=u_0<u_1<\ldots<u_{K(m,n)-1}< u_{K(m,n)}=b$. A natural estimator of the structural expectation ${f_{ES}}$ is then obtained by linear interpolation between the points $(u_k,v_k)$. For all $t\in[a,b]$, let define $$\widehat{{f_{ES}}}(t)=\sum_{k=0}^{K(m,n)-1}\left(v_k+\frac{v_{k+1}- v_k}{u_{k+1}-u_k}(t-u_k)\right){\bf 1}_{[u_k,u_{k+1})}(t)+v_{K(m,n)}{\bf 1}_{\{b\}}(t).$$ Note that by construction, this estimator is strictly increasing and continuous on $[a,b]$. The following theorem states its consistency.
\begin{theorem}[Consistency of the estimator of the Structural Expectation]\label{t_f}
Under Assumption \eqref{d_h},  we have 
$$\left\Vert\widehat{f_{ES}}- f_{ES}\right\Vert_\infty\xrightarrow[m,n\to\infty]{as}0.$$
\end{theorem}

Obtaining confidence bands for ${f_{ES}^{-1}}$ is useful for describing and visualizing the uncertainty in the estimate of ${f_{ES}^{-1}}$. This requires finding an estimator of $\mbox{var}(G(y))=\gamma\circ f^{-1}(y)-\left\{{f_{ES}^{-1}}(y)\right\}^2$, $y\in[f(a), f(b)]$.
\begin{lemma}\label{lcb}
Let $y\in[f(a), f(b)]$. Let $\widehat{\gamma\circ f^{-1}}(y)=\frac{1}{m}\sum_{i=1}^{m}T_i^2(y)$, with $T_i(.)$ defined as in (\ref{approxt}). Then,
\begin{equation*}
\widehat{\gamma\circ f^{-1}}(y)-\left\{\widehat{{f_{ES}^{-1}}}(y)\right\}^2 \xrightarrow[m,n\to\infty]{as}\mbox{var}(G(y)).
\end{equation*}
\end{lemma}
Proof of this lemma is given in Section 4.2. Combining this lemma with the asymptotic normality result stated by Theorem \ref{t_fm1} yields a pointwise asymptotic confidence band for ${f_{ES}^{-1}}$.
\begin{corollary}
An asymptotic $(1-\alpha)$-level pointwise confidence band for ${f_{ES}^{-1}}$ is given by
\begin{equation*}
\left[\widehat{{f_{ES}^{-1}}}(y)-u_{1-\frac{\alpha}{2}} \sqrt{\frac{\hat{\mbox{var}}(G(y))}{m}},\widehat{{f_{ES}^{-1}}}(y)+u_{1-\frac{\alpha}{2}} \sqrt{\frac{\hat{\mbox{var}}(G(y))}{m}}\right],
\end{equation*}
where $\hat{\mbox{var}}(G(y))=\widehat{\gamma\circ f^{-1}}(y)-\left\{\widehat{{f_{ES}^{-1}}}(y)\right\}^2$ and $u_{1-\frac{\alpha}{2}}$ is the quantile of order $1-\frac{\alpha}{2}$ of the standard normal distribution.
\end{corollary}
Note that the construction of a simultaneous asymptotic confidence band for ${f_{ES}^{-1}}$ would require determination of the distribution of $\sup_{f(a)\leq y\leq f(b)}|G(y)|$. This, however, falls beyond the scope of this paper.

\subsection{Estimator of an individual warping function}

In a warping framework, it is necessary to estimate the mean pattern but also the individual warping functions, i.e. $(h_i^{-1})_{i\in{\{1,\dots,m\}}}$. As previously, we can not directly estimate the functions $h_i^{-1}$ only the functions $\phi\circ h_i^{-1}$.\\ 
\indent   Let $i_0\in{\{1,\dots,m\}}$. Now we want to compute for all $i\neq i_0$, $T_i^\star(t)=f_i^{-1}\circ f_{i_0}(t)$. For this, define
\begin{equation}\label{deft2}
j_0(t)=\arg\min_{j \in{\{1,\dots,n \}}}\left|t_{i_0j}-t\right|.
\end{equation}
This point is the obervation time for the $i_0$ curve which is the closest to $t$. Note that the index $j_0(t)$ depends on $i_0$ but, for sake of simplicity we drop this index in the notations.
Then  $ \forall t\in[a,b]$, and $\forall i\in{\{1,\dots,m\}}\smallsetminus i_0$, compute
\begin{equation}\label{deft1}
T_i(t)=\arg\min_{t_j\in{\{t_{i1},\dots,t_{in}\}}}\left|Y_{ij}-Y_{i_0j_0(t)}\right|,
\end{equation}
an estimate of $T_i^\star$. Then for a fixed $i_0$, noting that $T^\star_i=h_i \circ h_{i_0}^{-1}$, we can see that an empirical estimator of each individual warping function $\phi \circ h_{i_0}^{-1}$ is given by
\begin{equation}\label{e_h}
\widehat{\phi\circ h_{i_0}^{-1}}(t):=\frac{1}{m-1}\sum_{\begin{subarray}{c}i=1\\i\neq i_0\end{subarray}}^{m}T_i(t).
\end{equation}
The following theorem asserts consistency and asymptotic normality of this estimator.
\begin{theorem}\label{t_h}
Under assumption \eqref{d_h},  
$$\left\Vert\widehat{\phi\circ h_{i_0}^{-1}}-\phi\circ h_{i_0}^{-1}\right\Vert_\infty\xrightarrow[m,n\to\infty]{as}0.$$
Let $n=m^{\frac{1}{2}+\alpha}$ (with $\alpha>0$) and assume that \eqref{h_1} and \eqref{h_3} hold. Then $\sqrt{m}(\widehat{\phi\circ h_{i_0}^{-1}}-\phi\circ h_{i_0}^{-1})$ converges weakly to a centered Gaussian process $Z$,
$$\sqrt{m}(\widehat{\phi\circ h_{i_0}^{-1}}-\phi\circ h_{i_0}^{-1})\xrightarrow[m\to\infty]{\D}Z,$$ with covariance function defined for all $(s,t)\in[a,b]^2$ by
$${\rm cov}(Z(s),Z(t))=\rr\left(h_{i_0}^{-1}(s),h_{i_0}^{-1}(t)\right).$$
\end{theorem}
We may also compute confidence bands for $\phi\circ h_{i_0}^{-1}$, based on a consistent estimator of $\mbox{var}(Z(t))=\gamma\circ h_{i_0}^{-1}(t)-\{\phi\circ h_{i_0}^{-1}(t)\}^2$.
\begin{lemma}\label{lcb2}
Let $t\in[a,b]$. Let $\widehat{\gamma\circ h_{i_0}^{-1}}(t)=\frac{1}{m}\sum_{i=1}^{m}T_i^2(t)$, with $T_i(.)$ defined by (\ref{deft1}) and (\ref{deft2}). Then
\begin{equation*}
\widehat{\gamma\circ h_{i_0}^{-1}}(t)-\left\{\widehat{\phi\circ h_{i_0}^{-1}}(t)\right\}^2 \xrightarrow[m,n\to\infty]{as}\mbox{var}(Z(t)).
\end{equation*}
\end{lemma}
Proof of this lemma relies on the same arguments as proof of Lemma \ref{lcb} and is  outlined in Section 4.2. A pointwise asymptotic confidence band for $\phi\circ h_{i_0}^{-1}$ is now given by
\begin{corollary}
An asymptotic $(1-\alpha)$-level pointwise confidence band for $\phi\circ h_{i_0}^{-1}$ is given by
\begin{equation*}
\left[\widehat{\phi\circ h_{i_0}^{-1}}(t)-u_{1-\frac{\alpha}{2}} \sqrt{\frac{\hat{\mbox{var}}(Z(t))}{m}},\widehat{\phi\circ h_{i_0}^{-1}}(t)+u_{1-\frac{\alpha}{2}} \sqrt{\frac{\hat{\mbox{var}}(Z(t))}{m}}\right],
\end{equation*}
where $\hat{\mbox{var}}(Z(t))=\widehat{\gamma\circ h_{i_0}^{-1}}(t)-\left\{\widehat{\phi\circ h_{i_0}^{-1}}(t)\right\}^2$.
\end{corollary}
\section{Extensions to the general case} \label{s:cmieux}
In the preceding part, we studied the asymptotic behaviour of a new warping methodology. However, drastic restrictions over the class of functions are needed, mostly dealing with monocity of the observed functions and on a non noisy model. In this part, we get rid of such asumptions and provide a pratical way of handling  more realistic observations.\vskip .1in
First, note that the assumptions $H(a)\stackrel{a.s}{=}a$ et $H(b)\stackrel{a.s}{=}b$ can be weakened into the following assumptions
\begin{itemize}
\item [ii$^\prime$)] $H^{-1}(\cdot,a)$ are $H^{-1}(\cdot,b)$ random variables compactly supported such that
$$\sup_{w\in\Omega}H^{-1}(w,a)\stackrel{a.s}{<}
\inf_{w\in\Omega}H^{-1}(w,b).$$
\end{itemize}
Then, we focus on the other assumptions.
\subsection{Breaking monotonicity}
The main idea is to build a transformation $\mathcal{G}$ which turns a nonmonotone function into a monotone function while preserving the warping functions. For sake of simplicity, the observation times will be taken equal for all the curves, hence $t_{ij}$ will be denoted $t_j$. Hence, the observations
$$Y_{ij}=f\circ f_i^{-1}(t_j),\:i=1,\dots,m,\:j=0,\cdots,n,$$
are transformed into 
\begin{equation}\label{modelg}
Z_{ij}=\mathcal{G}(f)\circ h_i^{-1}(t_j):=g\circ h_i^{-1}(t_j),\:i=1,\dots,m,\:j=0,\cdots,n,
\end{equation}
where $g$ is a monotone function. So, estimating the warping process of the monotonized model can be used to  estimate the real warping functions, and then align the original observations $Y_{ij}$ to their structural mean.\vskip .1in
For this, consider a nonmonotone function $f:[a,b]\to {\{1,\dots,m\}}$ and let 
$a=s_0<s_1<\dots<s_r<s_{r+1}=b$
be the different variational change points, in the sense  that $\forall k\in\{0,\dots,r-1\}$,
\begin{equation}\label{assumption_f2}
\begin{array}{c}
\forall(t_1,t_2)\in]s_k,s_{k+1}[,\:\forall(t_3,t_4)\in]s_{k+1},s_{k+2}[,\\
t_1<t_2\mbox{ et }t_3<t_4\Rightarrow\left(f(t_1)-f(t_2)\right)
\left(f(t_3)-f(t_4)\right)<0.
\end{array}
\end{equation}
So, consider functional warping over the subset
$$\mathcal{F}=\{f:[a,b]\to {\{1,\dots,m\}} \subset\mathbb{R}\mbox{ such that~\eqref{assumption_f2} holds }\}.$$
Let $\pi: ]a,b[\smallsetminus\{s_1,\dots,s_r\} \rightarrow \{-1,1\}$ be a tool function which indicates whether, around a given point $t$, the function is $f$ increasing or decreasing, defined by 
$$
\pi:s_{l(t)}<t<s_{l(t)+1}\mapsto \pi(t)=
\left\{\begin{array}{rcl}
-1&\mbox{ si }&s_{l(t)}-s_{l(t)+1}>0,\\
1&\mbox{ si }&s_{l(t)}-s_{l(t)+1}<0,
\end{array}\right.$$
with $l(t)\in\{0,\dots,r\}$.
\begin{description}
\item[Monotonizing Operator] 
\end{description}
Define for all $f\in \mathcal{F}$ the operator $\mathcal{G}(.,f):\: t\in]a,b[\smallsetminus\{s_0,\dots,s_{r+1}\} \to \mathcal{G}(t,f)$  by
$$\mathcal{G}(t,f)=
f(t)\pi(t)-\sum_{k=0}^{r}\pi(t)f(s_k)
{\bf 1}_{]s_k,s_{k+1}[}(t)+f(s_0)+\sum_{k=1}^{r}|f(s_{k-1})-f(s_k)|
{\bf 1}_{]s_k,b[}(t),$$
and, for all $k\in\{0,\dots,r+1\}$, 
$$\mathcal{G}(s_k,f)=f(a)+\sum_{l=1}^{k}\left|f(s_{l-1})-f(s_l)\right|,$$
with the notation $\sum_{l=1}^{0}\left|f(s_{l-1})-f(s_l)\right|=0$.\vskip .1in
By construction, it is obvious that $t \to \mathcal{G}(t,f)$ is strictly increasing.
Moreover, the following proposition proves that the warping functions remain unchanged 
\begin{proposition}\label{proposition_g2}
Consider $f\in\mathcal{F}$ and warping functions 
$h_i, \:{i=1,\dots,m}$. Set $\forall i=1,\dots,m,\: \mathcal{G}(.,f_i)=g_i(.)$. We have 
 $g_i=g\circ h_i^{-1}.$
\end{proposition}
\begin{proof}
For $i\in\{1,\dots,m\}$, let us prove that  $f_i \in\mathcal{F}$. Using \eqref{assumption_f2}, consider the change points  $\left(s_k\right)_{k=1,\dots,r}$ and set $s_k^i=h_i(s_k).$
For $k\in\{0,\dots,r-1\}$, consider
$$\left.(t_1,t_2)\in\left]s_k^i,s_{k+1}^i\right[,\:
(t_3,t_4)\in\left]s_{k+1}^i,s_{k+2}^i\right[\:\right/
\:t_1<t_2\mbox{ et }t_3<t_4,$$
then,
$$\left(f_i(t_1)-f_i(t_2)\right)\left(f_i(t_3)-f_i(t_4)\right)
=\left(f\circ h_i^{-1}(t_1)-f\circ h_i^{-1}(t_2)\right)
\left(f\circ h_i^{-1}(t_3)-f\circ h_i^{-1}(t_4)\right).$$
Since $h_i^{-1}$ is strictly increasing, we get $s_k<h_i^{-1}(t_1)<h_i^{-1}(t_2)<s_{k+1}$
and
$s_{k+1}<h_i^{-1}(t_3)<h_i^{-1}(t_3)<s_{k+2}.$ Hence 
$$\left(f_i(t_1)-f_i(t_2)\right)\left(f_i(t_3)-f_i(t_4)\right)<0,$$ which yelds that $f_i\in\mathcal{F}$. So define $g_i=\mathcal{G}(f_i)$. For all $t\in]a,b[\smallsetminus\{s_1^i,\dots,s_r^i\}$, we get
$$g_i(t)=f_i(t)\Pi(t,f_i)-\sum_{k=0}^{r}\Pi(t,f_i)f_i(s_k^i)
{\bf 1}_{]s_k^i,s_{k+1}^i[}(t)+f_i(s_0)+\sum_{k=1}^{r}
\left|f_i(s_{k-1}^i)-f_i(s_k^i)\right|{\bf 1}_{]s_k^i,b[}(t),$$
But
\begin{align*}
f_i(t)&=f\circ h_i^{-1}(t),\\
\Pi(t,f_i)&=\Pi(t,f\circ h_i^{-1})=\Pi(h_i^{-1}(t),f),\\
{\bf 1}_{]s_k^i,s_{k+1}^i[}(t)
&={\bf 1}_{]h_i(s_k),h_i(s_{k+1})[}(t)
={\bf 1}_{]s_k,s_{k+1}[}\left(h_i^{-1}(t)\right),
\end{align*}
which implies that
$$g_i=g\circ h_i^{-1}.$$
\end{proof}
The discretization implies however that the $Z_{ij}$ can not be computed directly since the functions $f_i,\:i=1,\dots,m,$ are known on the grid $t_j,\:{j=0,\dots,n}$, while the values,
$s_k$  and $s_k^i$ are unknown. So consider estimates of $Z_{ij}$ defined as follows
$$\tilde{Z}_{i0}=Y_{i0}$$
\begin{equation}\label{definition_Z}
\forall j\in\{1,\dots,n\},\:
\tilde{Z}_{ij}=\tilde{Z}_{ij-1}+\left|Y_{ij}-Y_{ij-1}\right|.
\end{equation}
The following proposition proves the consistency of such estimation procedure.
\begin{proposition}\label{proposition_Z}
For $f\in\mathcal{F}$,  $i\in\{1,\dots,m\}$ and $t\in[a,b]$, define a sequence $j(n)$ such that $\frac{j(n)}{n}\xrightarrow[n\to+\infty]{}t.$
Then,
$$\tilde{Z}_{ij(n)}-Z_{ij(n)}\xrightarrow[n\to+\infty]{a.s}0.$$
\end{proposition}
\begin{proof}
Set $t\in]a,b[\smallsetminus\left\{s_1^i,\dots,s_r^i\right\}$, and $l\in\{0,1,\dots,r\}$
such that $t\in\left]s_l^i,s_{l+1}^i\right[$. Consider $j(n), \: \frac{j(n)}{n}\xrightarrow[n\to+\infty]{}t,$
then $\exists n_0\in\mathbb{N}\:/\:\forall n\in\mathbb{N},\:n\geq n_0\Rightarrow
\frac{j(n)}{n}\in\left]s_l^i,s_{l+1}^i\right[.$
$\forall n\geq n_0$, we have
\begin{align*}
\tilde{Z}_{ij(n)}&=\tilde{Z}_{ij(n)-1}+\left|Y_{ij(n)}-Y_{ij(n)-1}\right|\\
&=Y_{i0}+\sum_{k=1}^{j(n)}\left|Y_{ik}-Y_{ik-1}\right|.
\end{align*}
Moreover,
\begin{align*}
Z_{ij(n)}&=g\circ h_i^{-1}\left(t_{j(n)}\right)\\
&=f_i\left(t_{j(n)}\right)\Pi\left(t_{j(n)},f_i\right)+\Pi
\left(t_{j(n)},f_i\right)
f_i\left(s_l^i\right)+f_i(s_0)+\sum_{k=1}^{l}\left|
f_i\left(s_{k-1}^i\right)-f_i\left(s_k^i\right)\right|,
\end{align*}
which yelds that
\begin{align*}
Z_{ij(n)}&=\Pi\left(t_{j(n)},f_i\right)
\left(Y_{ij(n)}-f_i\left(s_l^i\right)\right)+Y_{i0}+
\sum_{k=1}^{l}\left|f_i\left(s_{k-1}^i\right)-f_i\left(s_k^i\right)\right|\\
&=\left|Y_{ij(n)}-f_i\left(s_l^i\right)\right|+Y_{i0}+
\sum_{k=1}^{l}\left|f_i\left(s_{k-1}^i\right)-f_i\left(s_k^i\right)\right|\\
&=A+Y_{i0}+B.
\end{align*}
Write $\forall k=1,\dots,l+1,\: s_{k-1}^i\leq t_{j_k-p_k}<\cdots<t_{j_k}\leq s_k^i,$ we get that
\begin{align*}
&\left|f_i\left(s_{k-1}^i\right)-f_i\left(s_k^i\right)\right|\\
=&\left|f_i\left(s_{k-1}^i\right)-f_i\left(t_{j_k-p_k}\right)+
\sum_{q=1}^{p_k}\left(f_i\left(t_{j_k-q}\right)-f_i\left(t_{j_k-q+1}
\right)\right)+f_i\left(t_{j_k}\right)-f_i\left(s_k^i\right)\right|\\
=&\left|f_i\left(s_{k-1}^i\right)-f_i\left(t_{j_k-p_k}\right)\right|+
\sum_{q=1}^{p_k}\left|f_i\left(t_{j_k-q}\right)-f_i\left(t_{j_k-q+1}
\right)\right|+\left|f_i\left(t_{j_k}\right)-f_i\left(s_k^i\right)\right|\\
=&\left|f_i\left(s_{k-1}^i\right)-Y_{ij_k-p_k}\right|+
\sum_{q=1}^{p_k}\left|Y_{ij_k-q+1}-Y_{ij_k-q}\right|+
\left|Y_{ij_k}-f_i\left(s_k^i\right)\right|,
\end{align*}
Hence
\begin{align*}
B=&\sum_{k=1}^{l}\left|f_i\left(s_{k-1}^i\right)-f_i\left(s_k^i\right)\right|\\
=&\sum_{k=1}^{l}\left|f_i\left(s_{k-1}^i\right)-Y_{ij_k-p_k}\right|
+\sum_{k=1}^{l}\left|Y_{ij_k}-f_i\left(s_k^i\right)\right|+
\sum_{k=1}^{l}\sum_{q=1}^{p_k}|Y_{ij_k-q+1}-Y_{ij_k-q}|\\
=&\sum_{k=1}^{l}\left|f_i\left(s_{k-1}^i\right)-Y_{ij_k-p_k}\right|
+\sum_{k=1}^{l}\left|Y_{ij_k}-f_i\left(s_k^i\right)\right|+
\sum_{k=1}^{j_l}|Y_{ik}-Y_{ik-1}|\\
&-\sum_{k=1}^{l-1}|Y_{ij_{k+1}-p_{k+1}}-Y_{ij_{k+1}-p_{k+1}-1}|-
|Y_{i1}-Y_{i0}|.
\end{align*}
With the same ideas, we can write
\begin{align*}
&A=\left|Y_{ij(n)}-f_i\left(s_l^i\right)\right|=\\
&\sum_{q=j_{l+1}-j(n)+1}^{p_{l+1}+1}|Y_{ij_{l+1}-q+1}-Y_{ij_{l+1}-q}|+
\left|Y_{ij_{l+1}-p_{l+1}}-f_i\left(s_l^i\right)\right|-
|Y_{ij_{l+1}-p_{l+1}}-Y_{ij_{l+1}-p_{l+1}-1}|.
\end{align*}
As a result,
\begin{align*}
Z_{ij(n)}-\tilde{Z}_{ij(n)}=&
\sum_{k=1}^{l+1}\left|f_i\left(s_{k-1}^i\right)-Y_{ij_k-p_k}\right|+
\sum_{k=1}^l\left|Y_{ij_k}-f_i\left(s_k^i\right)\right|\\
&-\sum_{k=1}^l|Y_{ij_{k+1}-p_{k+1}}-Y_{ij_{k+1}-p_{k+1}-1}|
-|Y_{i1}-Y_{i0}|.
\end{align*}
By continuity of $f$, $f_i$ is also continuous, hence
$$Z_{ij(n)}-\tilde{Z}_{ij(n)}\xrightarrow[n\to+\infty]{a.s}0.$$
For $t\in\left\{s_0^i,\dots,s_{r+1}^i\right\}$, we get similar results, leading to the conclusion.
\end{proof}
As a conclusion, we can extend our results to the case of nonmonotone functions since we transform the problem into a monotone warping problem with the same warping functions. These functions $h_i\: i=1,\dots,m$ can be estimated by our methodology using the new observations $\tilde{Z}_{ij}$ $ \forall i=1,\dots,m,\: \forall j=0,\dots,n$. The estimator can be written in the following form
\begin{equation}\label{estimator_hg}
\widetilde{\phi\circ h_{i_0}^{-1}}_{mn}(t)=\frac{1}{m-1}
\sum_{\begin{subarray}{c}i=1\\i\neq i_0\end{subarray}}^{m}T_{j_i},
\end{equation}
with 
$$T_{j_i}=\arg\min_{t_j\in\{t_0,\dots,t_n\}}
\left|\tilde{Z}_{ij}-\tilde{Z}_{i_0j_0}\right|,$$
and
$$t_{j_0}=\arg\min_{t_j\in\{t_0,\dots,t_n\}}|t_j-t|.$$
\subsection{Dealing with noisy data}
If theoretical asymptotical results are only given in a non noisy framework, we can still handle the case where the data are observed in the standard regression model
\begin{equation}\label{simulation_datab}
Y_{ij}=f\circ h_i^{-1}(t_j)+\varepsilon_{ij},\:i=1,\dots,m,\:j=0,\dots,n,
\end{equation}
with $\varepsilon_{ij} \stackrel{{\rm i.i.d}}{\sim} \mathcal{N}(0,\sigma^2)$. To be able to apply our algorithm, we first denoise the data. For this, we estimate separately each function $f_i,\:i=1,\dots,m,$ by a Kernel estimator. On a practical point of view, we describe the estimation procedure used in the simulations.
\begin{enumerate}
\item $\forall i\in\{1,\dots,m\}$, $f_i$ is estimated
$$\hat{f}_i(t_0)=\frac{1}{m}\sum_{i=1}^{m}Y_{i0}
\xrightarrow[m\to+\infty]{a.s}f(t_0),$$
$$\hat{f}_i(t_n)=\frac{1}{m}\sum_{i=1}^{m}Y_{in}
\xrightarrow[m\to+\infty]{a.s}f(t_n).$$ Given a Gaussian kernel $\Phi$, $\forall j\in\{1,\dots,n-1\}$, 
 $f_i(t_j)$ is estimated by
\begin{equation}\label{method_nu}
\hat{f}_i(t_j)=\frac{\sum_{k=0}^{n}Y_{ik}\Phi\left(
\frac{t_k-t_j}{\nu_i}\right)}{\sum_{k=0}^{n}\Phi\left(
\frac{t_k-t_j}{\nu_i}\right)}.
\end{equation}
The banwidths $\nu_i$ are to be properly chosen.
\item The estimation procedure can be conducted using the denoised observations $\hat{f}_i(t_j)$, leading to new estimates $\hat{f}(t)$ of the structural expectation $f\circ\phi^{-1}$.
\end{enumerate}
We point out that the efficiency of the procedure heavily relies on a proper choice of the bandwidths $\nu_i,\: i=1,\dots,m$. Cross-validation technics do not provide good results since the aim is not to get a good estimation of the function but only a good separation of the different functions. Hence oversmoothing the data is not a drawback in this settings. So, the smoothing parameters $\nu=\nu_i, \forall i=1,\dots,m$ are obtained by minimizing the following matching criterion
\begin{equation*}
\hat{\nu}=\arg\min_{\nu\in L}\sum_{i=1}^{m}\sum_{j=0}^{n}
\left|\hat{f}^i(t_j)-\hat{f}(t_j)\right|,
\end{equation*}
over a grid $L$. Practical applications of this algorithm are given in Figure
\section{Numerical Study}\label{s_data}
In this part, we estimate the structural expectation, using both the proposed method and the analytic registration approach developed in \cite{Ramsay02}. Note that an alternative to analytic registration is provided by the so-called landmarks registration approach in \cite{Kneip92}, but this approach requires the determination of landmarks (such as local extrema) which can be difficult in our simulations. Hence, this method is not implemented here. First, results on simulated data are given in order to compare these two methods. Then, an application of our methodology is given for a real data set.
\subsection{Simulations}
Two simulation studies are carried out in this section. The first one involves a strictly increasing function and the second one involves a non monotone function.
\paragraph{Warped functions.} Let $f$ and $g$ (see Figures \ref{fig01} and \ref{fig02}) be defined by
$$\forall t\in[0,1],\:f(t)=\sin(3\pi t)+3\pi t\mbox{ and }g(t)=\frac{\sin(6\pi t)}{6\pi t}.$$
These two functions will be warped by the following random warping process.
\paragraph{Warping processes.} The Stochastic warping functions $H_i$, $i=1,\dots,m$, are simulated by the iterative process described below.\\

\fbox{\begin{minipage}{0.9\textwidth}
Let $\epsilon>0$. First, for all $i=1,\dots,m$, let $H_i^{(0)}$ be the identity function. Then, the warping functions $H_i^{(k+1)}$, $i=1,\dots,m$, are successively carried out from functions $H_i^{(k)}$, $i=1,\dots,m$, by iterating $N$ times the following process :
\begin{enumerate}
\item Let $U$ be a uniformly distributed random variable on $[10\epsilon,1-10\epsilon]$.
\item Let $V_i$, $i=1,\dots,m$, be independant and identically uniformly distributed variables on $[U-\epsilon,U+\epsilon]$.
\item For all $i=1,\dots,m$, the warping function $H_i^{(k)}$ is warped as follows :
$$H_i^{(k+1)}=W_i\circ H_i^{(k)}$$
where $W_i$ is defined by
$$W_i(t)=\left\{\begin{array}{cl}
\frac{V_i}{U}t&\mbox{if }0\leq t\leq U,\\
\frac{1-V_i}{1-U}t+\frac{V_i-U}{1-U}&\mbox{if }U<t\leq 1.\\
\end{array}\right.$$
\end{enumerate}
\end{minipage}}
$ $ \vskip .1in
We have thus defined stricly increasing stochatic functions $H_i:[0,1]\to[0,1]$, such that
$$\forall t\in[0,1],\:{\bf E}\left(H_i(t)\right)=\phi(t)=t.$$
Hence, the warping process is centered, in the sense that, the structural expectation $f\circ\phi^{-1}$ is equal to $f$.\\

Our simulated warping functions $H_i$, $i=1,\dots,m$, have been carried out by using the above iterating process with $m=30$, $N=3000$ and $\epsilon=0.005$. These processes are shown on Figure \ref{fig01} and on Figure \ref{fig02}. In these figure, we can see the very large phase variations. For example, point 0.2 is warped between approximately 0.05 and 0.35 for the first case, and between approximately 0.00 and 0.40 for the second case.
\paragraph{Simulated data.} Finally, simulated data are carried out on an equally spaced grid as follows :
$$Y_{ij}=f\left(H_i^{-1}\left(t_j\right)\right)$$
and
$$Y_{ij}=g\left(H_i^{-1}\left(t_j\right)\right)+\epsilon_{ij}$$
with $t_j=\frac{j}{n}$, $j=0,1,\dots,n$, $n=100$ points. Simulated warped functions are respectively shown on Figures \ref{fig01} and \ref{fig02}. \vskip .2in
Figures \ref{fig01} and \ref{fig02} show the functions $f$ and $g$ and the mean functions of the warped functions. We can easily see that the classical mean is not adapted to describe the data. In particular, the mean function of the first simulated data does not reflect the flat component of the function $f$ (in the range $[0.2,0.4]$) which yet appears in each individual warped function. In the same way, the mean function of the second simulated data set does not reflect the structure of the individual curves.  The classical mean attenuated curve variations.\\

The estimated structural expectations with both the analytic registration approach and the proposed method are shown on Figures \ref{fig01} and \ref{fig02} (bottom right figures). We can easily see that estimations with our proposed method are closer to the structural expectations $f$ and $g$. These results can be explained as follows :
\begin{itemize}
\item For the first simulated data set, the analytic registration approach does not directly work on the strictly increasing functions but on first derivatives. However,  theoretically registering a given function data set is not the same issue as the registration of first derivatives of these functions.
\item The analytic registration approach uses the mean curve to register all the functions. Due to the drawbacks of the mean curve when dealing with large deformations (for instance in the second data set, where the result is a very flat mean curve), the structural mean approach provides better results.
\item For both first and second simulated data sets, the analytic registration approach works on estimated functions and not directly with given data, which  implies an additional source of error.
\end{itemize}
\begin{figure}[htbp]
        \centering
        \includegraphics[width=0.70\textwidth]{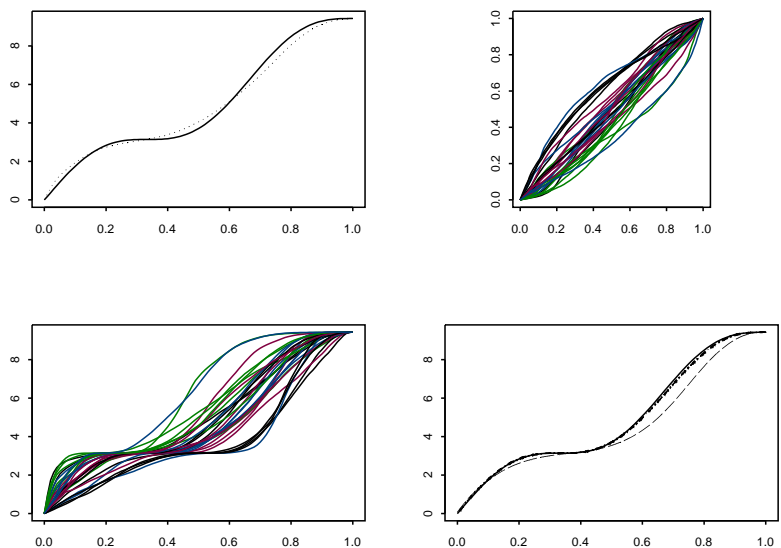}
        \caption{Function $f$ is shown on top left figure (solid line). Simulated warping processes are shown on top right figure. Simulated warped functions are shown on bottom left figure. The classical mean of these functions is draw on top left figure (dotted line). Finally, estimations of $f$ with the Analytic registration procedure (dashed line) and with our method (dotted-dashed line) are shown on bottom right figure.}
        \label{fig01}
\end{figure}
\begin{figure}[htbp]
        \centering
        \includegraphics[width=0.70\textwidth]{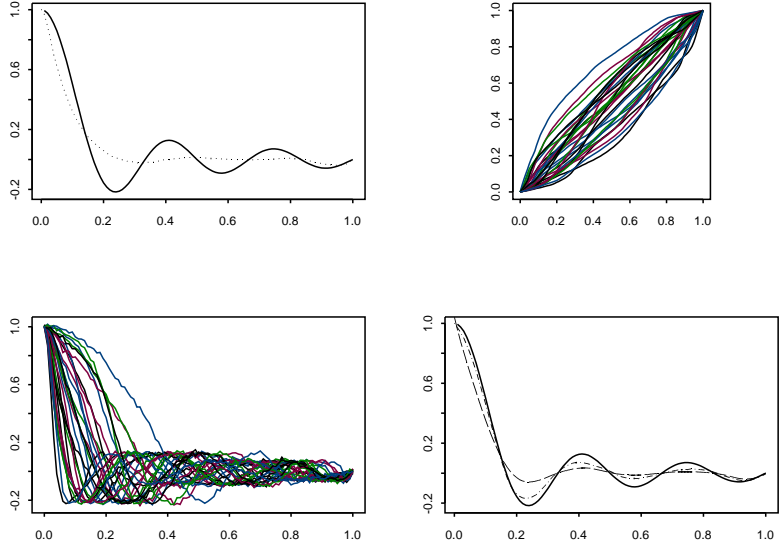}
        \caption{Function $g$ is shown on top left figure (solid line). Simulated warping processes are shown on top right figure. Simulated warped functions are shown on bottom left figure. The classical mean of these functions is draw on top left figure (dotted line). Finally, estimations of $g$ with the Analytic registration procedure (dashed line) and with our method (dotted-dashed line) are shown on bottom right figure.}
        \label{fig02}
\end{figure}
\subsection{A concrete application : multiple referees and equity}
The field of application of the results presented in this paper is large. Here we consider an example in the academic field : how to guarantee equality in an exam with several different referees?

Consider an examination with a large number of candidates, such that it is impossible to evaluate the candidates one after another. So the students are divised into $m$ groups and a board of examiners is charged to grade the students of one group. The evaluation is performed by assigning a score from $0$ to $20$.

The $m$ different boards of examiners are supposed to behave the same way in order to respect the equality among the candidates. Moreover it is assumed that the sampling of the candidates is perfect in the sense that it is done in such a way that each board of examiners evaluates candidates with the same global level. Hence, if all the examiners had the same requirement levels, the distribution of the ranks would be the same for all the boards of examiners. Here, we aim at balancing the effects of the differences between the examiners, and gaining equity for the candidates, and at studying real data provided by the french competitive teacher exam {\sl Agr\'egation de math\'ematiques}.\\

This situation can be modeled as follows. Define for each group $i=1,\dots,m$, the score obtained by students as ${\bf X}^i=\left\{X^i_j\in(0,20),\:j=1,\dots,n_i\right\}$. Let $f_i$, $i=1,\dots,m$ be the repartition function of the group $i$ defined by
$$f_i(t)=\frac{1}{n}\sum_{j=1}^n{\bf1}_{X^i_j\leq t}.$$
Under the assumption that all the examiners give the same ranking to two candidates with the same level, which will be denoted $H_0$, it implies that the random variables ${{\bf X}}^1,\dots,{\bf X}^m$ are equally distributed. Figure \ref{fig03} shows the distinct empirical functions for all the $m=13$ groups.\\

First, we test for every couple of sets $(i,j)$ the assumption ${\bf X}^i \sim {\bf X}^j$. For this we perform the following  homogeneity test, see for instance \cite{saporta}. Define for all $k=1,\dots,20$, $n_k^i=\sum_{l=1}^m {\bf 1}_{X^i_l=k}$ and $n_k^j=\sum_{l=1}^m {\bf 1}_{X^j_l=k}$. Define $\hat{\mu}_k=\frac{n_k^i+n_k^j}{2n}.$ Finally set $d_k= n \hat{\mu}_k$ and
$$D_n^i =\sum_{k=1}^{20} \frac{(d_k-n_k^i)^2}{d_k}, \quad D_n^j =\sum_{k=1}^{20} \frac{(d_k-n_k^j)^2}{d_k}.$$
Under $H_0$, $D_n=D_n^i+D_n^j$ converges in distribution to a $\chi^2$ distribution and almost surely to $+\infty$ under the assumption that the two laws are different.\\

We have $n=4000$ candidates and $m=13$ examiners. We test the assumption {\bf $H_0$} for all the different possibilities. In 60\% of cases and for a level equal to $5\%$, we reject the assumption that the rankings follow the same distribution.\\
 
As a consequence the following procedure is proposed. We aim at finding the average way of ranking, with respect to the ranks that were given within the $13$ bunchs of candidates. For this, assume that there is such average empirical distribution function and that the empirical distribution function of each group is warped by a random process from this reference distribution function. Hence a good choice is given by the structural expectation, since the functions $f_j,\:j=1,\dots,13$ are increasing so that Theorem~\ref{t_fm1} may apply.\\
Figure \ref{fig03} shows the different empirical distributions and the structural empirical distribution function. At each dot on the empirical distribution function, corresponding to the rank given within one group, is associated its correspondent structural rank. It is obtained by simply projecting onto the structural empirical distribution function. As a result we obtain rescaled structural ranks corresponding to the rank obtained by each candidate if they could have been judged by an average board of examiner, leading to a more fair rankings. Indeed the difference of judgments between each subgroup are revised according to the average judgment of all the groups of examiners.\\
\begin{figure}[htbp]
        \centering
        \includegraphics[width=0.70\textwidth]{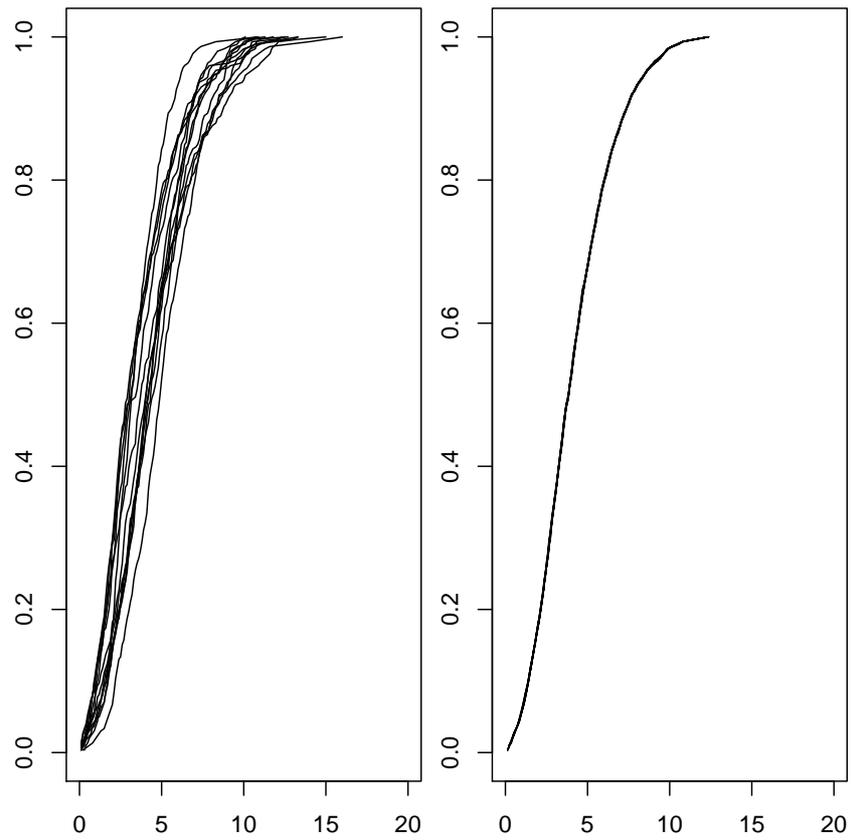}
        \caption{Empirical distribution functions (left figure) and structural expectation estimation (right figure).}
        \label{fig03}
\end{figure}
In conclusion, structural expectation provides a data-driven pattern, which plays the role of a reference pattern to which all the different curves can be compared. We applied successfully this method to rescale the ranks obtained by candidates evaluated by different boards of examiners. This use is not restrictive since it can be used to provide mean patterns for other types of functional data in various fields such as econometry, biology or data traffic for instance.

\appendix
\section{Proofs and technical lemmas}\label{s_proofs}

In practice, functions $(f_i)_{i\in{\{1,\dots,m\}}}$ are evaluated on a discrete interval of $\mathbb{R}$ as described in Section \ref{s_model}. In order to prove Theorem \ref{t_fm1} and Theorem \ref{t_h}, we first study asymptotic results on the theoretical continuous model, i.e.
\begin{equation}\label{m_continu}
f_i(t)=f\circ h_i^{-1}(t),\:i\in{\{1,\dots,m\}},\:t\in[a,b],
\end{equation}
where $H$ is defined in the same way as in model \eqref{m_model}. In a second step, we will extend the proofs to the discretized model and prove results of Section 3.\\
\indent So consider that all functions are measured on the entire interval $[a,b]$. After asymptotic results are proved for this continuous model in Section \ref{ss_proofs_continu}, we use these results to prove Theorem \ref{t_fm1} and Theorem \ref{t_h} (and subsequently Theoreme \ref{t_f}) in Section \ref{ss_proofs_proofs}.

\subsection{Asymptotic results for the continuous model}\label{ss_proofs_continu}

For the continuous model \eqref{m_continu}, we provide asymptotic results (analogous to Theorem \ref{t_fm1} and Theorem \ref{t_h}) and their proofs.

\subsubsection{Estimator and asymptotic results of the inverse of the structural expectation}

Considering the continuous model \eqref{m_continu}, we define an empirical estimator of the inverse of the structural expectation in the following way. Set
\begin{equation}\label{e_fm1_continu}
\overline{{f_{ES}^{-1}}}=\frac{1}{m}\sum_{i=1}^{m}f_i^{-1}.
\end{equation}
The following theorem gives us consistency and asymptotic normality of this estimator.
\begin{theorem}\label{t_fm1_continu}
Under assumption \eqref{d_h}, we have that $\overline{{f_{ES}^{-1}}}$ converges almost surely to ${f_{ES}^{-1}}$
$$\left\Vert\overline{{f_{ES}^{-1}}}-{f_{ES}^{-1}}\right\Vert_\infty\xrightarrow[m\to\infty]{as}0.$$
Moreover, let assume that assumptions \eqref{h_1} and \eqref{h_2} are allowed. Then, we have that $\sqrt{m}(\overline{{f_{ES}^{-1}}}-{f_{ES}^{-1}})$ converges weakly toward a zero-mean Gaussian process $G$:
$$\sqrt{m}\left(\overline{{f_{ES}^{-1}}}-{f_{ES}^{-1}}\right)\xrightarrow[m\to\infty]{\D}G,$$
where the covariance function of $G$ is defined, for all $(s,t)\in[f(a),f(b)]^2$, by
$${\rm Cov}(G(s),G(t))=\rr\left(f^{-1}(s),f^{-1}(t)\right).$$
\end{theorem}
\begin{proof}
Almost sure convergence of estimator $\overline{{f_{ES}^{-1}}}$ is directly deduced from corollary 7.10 (p. 189) in \cite{Ledoux91}. This corollary is an extension of the Strong Law of Large Numbers to Banach spaces.\\
For all $i\in{\{1,\dots,m\}}$, the functions $\left(f_i^{-1}\right)_{i\in{\{1,\dots,m\}}}$ are obviously strictly increasing, hence $\overline{{f_{ES}^{-1}}}$ is strictly increasing, and we have
\begin{equation*}
\overline{{f_{ES}^{-1}}}=\frac{1}{m}\sum_{i=1}^{m}f_i^{-1}=\frac{1}{m}\sum_{i=1}^{m}\left(f\circ h_i^{-1}\right)^{-1}=\frac{1}{m}\sum_{i=1}^{m}h_i\circ f^{-1}.
\end{equation*}
For all $i\in\mathbb{N}^*$, let
$$X_i=h_i\circ f^{-1}-{f_{ES}^{-1}},$$
and, for all $m\in\mathbb{N}^*$, let
$$S_m=\sum_{i=1}^{m}X_i.$$
The $\left(X_i\right)_{i\in{\{1,\dots,m\}}}$ are $B$-valued random variables, where $B=\mathcal{C}\left([f(a),f(b)]\right)$ is a separable Banach space. Moreover, the dual space of $B$ is the set of bounded measures on $\left[f(a),f(b)\right]$ (\cite{Rudin87}). Hence, our framework corresponds to Chapter 7 of \cite{Ledoux91} and thus we can apply Corollary 7.10. Indeed, we have
$${\bf E}\left(\Vert X_1\Vert_\infty\right)<+\infty\mbox{ et }{\bf E}\left(X_1\right)=0,$$
then
$$\frac{S_m}{m}\xrightarrow[m\to+\infty]{a.s}0,$$
which proves almost sure convergence.

We now turn to weak convergence. From the multivariate CLT, for any $k\in\mathbb N^{\ast}$ and fixed $(y_1,\ldots, y_k)\in \left[f(a),f(b)\right]^k$,
\begin{eqnarray*}
\sqrt m
\left(\left(
\begin{array}{c}
\overline{{f_{ES}^{-1}}}(y_1) \\ \vdots \\ \overline{{f_{ES}^{-1}}}(y_k)
\end{array}\right)
-
\left(
\begin{array}{c}
{f_{ES}^{-1}}(y_1) \\ \vdots \\ {f_{ES}^{-1}}(y_k)
\end{array}
\right)
\right)
\xrightarrow[m\to\infty]{\mathcal D}\mathcal N_k\left(0,\Gamma\right),
\end{eqnarray*}
where the covariance matrix $\Gamma=(\Gamma_{ij})_{i,j}$ is given by $\Gamma_{ij}=\mbox{cov}(F^{-1}(y_i), F^{-1}(y_j))= \mbox{cov}(H(f^{-1}(y_i)),H(f^{-1}(y_j)))$, $i,j=1,\ldots, k$. It remains to show that $\{\sqrt m(\overline{{f_{ES}^{-1}}}-{f_{ES}^{-1}})\}$ is tight. We verify the moment condition stated by \cite{VaartWellner96} [Example 2.2.12].
\begin{eqnarray*}
&&{\bf E}\left[\left|\sqrt m(\overline{{f_{ES}^{-1}}}(s)-{f_{ES}^{-1}}(s))-
\sqrt m(\overline{{f_{ES}^{-1}}}(t)-{f_{ES}^{-1}}(t))\right|^2\right]\\
&&\hspace{2cm}={\bf E}\left[m\left|\frac{1}{m}\sum_{i=1}^mf_i^{-1}(s)-{\bf E} F^{-1}(s)-\left(
\frac{1}{m}\sum_{i=1}^mf_i^{-1}(t)-{\bf E} F^{-1}(t)\right)\right|^2\right]\\
&&\hspace{2cm}={\bf E}\left[\left|F^{-1}(s)-{\bf E} F^{-1}(s)-\left(F^{-1}(t)-{\bf E} F^{-1}(t)\right)\right|^2\right],
\end{eqnarray*}
where the last equality follows from the fact that the $h_i$'s are i.i.d. Then, from \eqref{h_1} and \eqref{h_2}, we get that
\begin{eqnarray*}
{\bf E}\left[\left|\sqrt m(\overline{{f_{ES}^{-1}}}(s)-{f_{ES}^{-1}}(s))-
\sqrt m(\overline{{f_{ES}^{-1}}}(t)-{f_{ES}^{-1}}(t))\right|^2\right]\leq C_1C_2|s-t|^2,
\end{eqnarray*}
which completes the proof.
\end{proof}

\subsubsection{Estimator and asymptotic results of an individual warping function}

For the continuous model, we define an empirical estimator of the individual warping function $\phi\circ h_{i_0}^{-1}$ ($i_0\in {\{1,\dots,m\}}$) as follows. Conditionally to $F_{i_0}=f_{i_0}$, for all $t\in[a,b]$, let
\begin{equation}\label{e_h_continu}
\overline{\phi\circ h_{i_0}^{-1}}(t)=\frac{1}{m-1}\sum_{\begin{subarray}{c}i=1\\i\neq i_0\end{subarray}}^{m}f_i^{-1}\circ F_{i_0}(t).
\end{equation}
The following theorem gives us consistency and asymptotic normality of this estimator.
\begin{theorem}\label{t_h_continu}
Under assumption \eqref{d_h}, we have that $\overline{\phi\circ h_{i_0}^{-1}}$ converges almost surely to $\phi\circ h_{i_0}^{-1}$:
$$\left\Vert\overline{\phi\circ h_{i_0}^{-1}}-\phi\circ h_{i_0}^{-1}\right\Vert_\infty\xrightarrow[m\to\infty]{as}0.$$
Let $n=m^{\frac{1}{2}+\alpha}$ (with $\alpha>0$) and assume that \eqref{h_1} and \eqref{h_3} hold. Then, we have that  $\sqrt{m}(\overline{\phi\circ h_{i_0}^{-1}}-\phi\circ h_{i_0}^{-1})$ converges weakly to a zero-mean Gaussian process $Z$,
$$ \sqrt{m}(\overline{\phi\circ h_{i_0}^{-1}}-\phi\circ h_{i_0}^{-1})\xrightarrow[m\to\infty]{\D}Z,$$ with covariance function defined for all $(s,t)\in[a,b]^2$ by
$$\cov(Z(s),Z(t))=\rr\left(h_{i_0}^{-1}(s),h_{i_0}^{-1}(t)\right).$$
\end{theorem}
\begin{proof}
Let $i_0\in{\{1,\dots,m\}}$. Given $F_{i_0}=f_{i_0}$,
\begin{equation*}
\overline{\phi\circ h_{i_0}^{-1}}=\frac{1}{m-1}\sum_{\begin{subarray}{c}i=1\\i\neq i_0\end{subarray}}^{m}f_i^{-1}\circ f_{i_0}.
\end{equation*}
Noting also that $\phi\circ h_{i_0}^{-1}={f_{ES}^{-1}}\circ f_{i_0}$, consistency of $\overline{\phi\circ h_{i_0}^{-1}}$ follows by the same arguments as in proof of Theorem \ref{t_fm1_continu}.

We now turn to weak convergence. From the multivariate CLT, for any $k\in\mathbb N^{\ast}$ and fixed $(t_1,\ldots, t_k)\in [a,b]^k$,
\begin{eqnarray*}
\sqrt m
\left(\left(
\begin{array}{c}
\overline{\phi\circ h_{i_0}^{-1}}(t_1) \\ \vdots \\ \overline{\phi\circ h_{i_0}^{-1}}(t_k)
\end{array}\right)
-
\left(
\begin{array}{c}
\phi\circ h_{i_0}^{-1}(t_1) \\ \vdots \\ \phi\circ h_{i_0}^{-1}(t_k)
\end{array}
\right)
\right)
\stackrel{\mathcal D}{\longrightarrow}\mathcal N_k\left(0,\Gamma_0\right),
\end{eqnarray*}
where the covariance matrix $\Gamma_0=(\Gamma_{0,ij})_{i,j}$ is given by $\Gamma_{0,ij}=\mbox{cov}(H(h_{i_0}^{-1}(t_i)), H(h_{i_0}^{-1}(t_j)))=r(h_{i_0}^{-1}(t_i), h_{i_0}^{-1}(t_j))$, $i,j=1,\ldots, k$. It remains to show that $\{\sqrt m(\overline{\phi\circ h_{i_0}^{-1}}-\phi\circ h_{i_0}^{-1})\}$ is tight. Again, we verify the moment condition stated by \cite{VaartWellner96} [Example 2.2.12].
\begin{eqnarray*}
&&{\bf E}\left[\left|\sqrt{m}(\overline{\phi\circ h_{i_0}^{-1}}(s)-\phi\circ h_{i_0}^{-1}(s))-\sqrt{m}(\overline{\phi\circ h_{i_0}^{-1}}(t)-\phi\circ h_{i_0}^{-1}(t))\right|^2\right]\\
&&=\frac{m}{(m-1)^2}{\bf E}\left[\left|\sum_{\begin{subarray}{c}i=1\\i\neq i_0\end{subarray}}^m\left(f_i^{-1}(f_{i_0}(s))-{\bf E} f_i^{-1}(f_{i_0}(s))-\left(f_i^{-1}(f_{i_0}(t))-{\bf E} f_i^{-1}(f_{i_0}(t))\right)\right)\right|^2\right]\\
&&=\frac{m}{m-1} {\bf E}\left[\left|H(h^{-1}_{i_0}(s))-{\bf E} H(h^{-1}_{i_0}(s))-\left(H(h^{-1}_{i_0}(t))-{\bf E} H(h^{-1}_{i_0}(t))\right)\right|^2\right].
\end{eqnarray*}
Now, from assumptions \eqref{h_1} and \eqref{h_3}, we get that
\begin{eqnarray*}
{\bf E}\left[\left|\sqrt{m}(\overline{\phi\circ h_{i_0}^{-1}}(s)-\phi\circ h_{i_0}^{-1}(s))-\sqrt{m}(\overline{\phi\circ h_{i_0}^{-1}}(t)-\phi\circ h_{i_0}^{-1}(t))\right|^2\right]\leq 2C_1C_3|s-t|^2,
\end{eqnarray*}
which completes the proof.
\end{proof}

\subsection{Proofs of asymptotic results}\label{ss_proofs_proofs}

We now use Theorem \ref{t_fm1_continu} (given for the continuous model) to prove Theorem \ref{t_fm1}.
\\\\
\textsc{Proof of Theorem \ref{t_fm1}} Let $y\in[f(a), f(b)]$. The $n$ observation times are equidistant and for each $i=1,\ldots,m$, $f_i$ is almost surely increasing, hence $f_i^{-1}(y)-\frac{1}{n}\leq T_i(y)\leq f_i^{-1}(y)+\frac{1}{n}$. This implies that almost surely,
\begin{equation}\label{ineg1}
\overline{{f_{ES}^{-1}}}(y)-\frac{1}{n}\leq \widehat{{f_{ES}^{-1}}}(y) \leq \overline{{f_{ES}^{-1}}}(y)+\frac{1}{n}.
\end{equation}
Since
\begin{equation*}
\left\Vert\overline{{f_{ES}^{-1}}}+\frac{1}{n}-{f_{ES}^{-1}}\right\Vert_{\infty}\leq \left\Vert \overline{{f_{ES}^{-1}}}-{f_{ES}^{-1}}\right\Vert_{\infty}+\frac{1}{n},
\end{equation*}
we get that $\left\Vert\overline{{f_{ES}^{-1}}}+\frac{1}{n}-{f_{ES}^{-1}}\right\Vert_{\infty} \xrightarrow[m,n\to\infty]{as}0$, by Theorem \ref{t_fm1_continu}. A similar argument holds for LHS(\ref{ineg1}) and finally, $\left\Vert\widehat{{f_{ES}^{-1}}}-{f_{ES}^{-1}}\right\Vert_{\infty} \xrightarrow[m,n\to\infty]{as}0$.

Let $n=m^{\frac{1}{2}+\alpha}$ $(\alpha>0)$. From (\ref{ineg1}), we get that almost surely,
\begin{equation*}
\left\Vert\sqrt m(\widehat{{f_{ES}^{-1}}}-\overline{{f_{ES}^{-1}}})\right\Vert_{\infty}\leq \frac{1}{m^{\alpha}}.
\end{equation*}
Since $\left\Vert\sqrt m(\widehat{{f_{ES}^{-1}}}-\overline{{f_{ES}^{-1}}})\right\Vert_{\infty}=\left\Vert\sqrt m(\widehat{{f_{ES}^{-1}}}-{f_{ES}^{-1}})-\sqrt m(\overline{{f_{ES}^{-1}}}-{f_{ES}^{-1}})\right\Vert_{\infty}$, we get that\\
$\left\Vert\sqrt m(\widehat{{f_{ES}^{-1}}}-{f_{ES}^{-1}})-\sqrt m(\overline{{f_{ES}^{-1}}}-{f_{ES}^{-1}})\right\Vert_{\infty}$ converges almost surely to 0 as $m$ tends to infinity. Combining Theorem \ref{t_fm1_continu} and Theorem 4.1 in \cite{Billingsley68}, it follows that $\sqrt m(\widehat{{f_{ES}^{-1}}}-{f_{ES}^{-1}}) \xrightarrow[m\to\infty]{\D}G$. 
\\
We now turn to proof of Lemma \ref{lcb}.
\\
\\
\textsc{Proof of Lemma \ref{lcb}}
Consider first the continuous model (\ref{m_continu}), and define
\begin{equation*}
\overline{\gamma\circ f^{-1}}=\frac{1}{m}\sum_{i=1}^m\left(f_i^{-1}\right)^2.
\end{equation*}
Using similar arguments as in proof of Theorem \ref{t_fm1_continu}, we get that $\left\Vert \overline{\gamma\circ f^{-1}}-\gamma\circ f^{-1}\right\Vert_{\infty}\xrightarrow[m\to\infty]{as}0.$

Now, since $|T_i(y)-f_i^{-1}(y)|\leq \frac{1}{n}$, we obtain by straightforward calculations that almost surely,
\begin{equation*}
-\frac{1}{n}\cdot\frac{1}{m}\sum_{i=1}^m|f_i^{-1}(y)|\leq\frac{1}{m}\sum_{i=1}^m T_i.f_i^{-1}(y)-\overline{\gamma\circ f^{-1}}(y)\leq \frac{1}{n}\cdot\frac{1}{m}\sum_{i=1}^m|f_i^{-1}(y)|,
\end{equation*}
which implies that $\frac{1}{m}\sum_{i=1}^m T_i.f_i^{-1}(y)-\overline{\gamma\circ f^{-1}}(y) \xrightarrow[m,n\to\infty]{as}0$, from which we deduce that $\frac{1}{m}\sum_{i=1}^m T_i.f_i^{-1}(y)\xrightarrow[m,n\to\infty]{as}\gamma\circ f^{-1}(y)$.

From $|T_i(y)-f_i^{-1}(y)|\leq \frac{1}{n}$, we also get that
\begin{equation*}
0\leq \frac{1}{m}\sum_{i=1}^m T_i^2-\frac{2}{m}\sum_{i=1}^m T_i.f_i^{-1}(y)+\frac{1}{m}\sum_{i=1}^m \left(f_i^{-1}(y)\right)^2\leq\frac{1}{n^2},
\end{equation*}
that is, $0\leq \widehat{\gamma\circ f^{-1}}(y)-\frac{2}{m}\sum_{i=1}^m T_i.f_i^{-1}(y)+\overline{\gamma\circ f^{-1}}(y)\leq\frac{1}{n^2},$ from which we deduce that that $\widehat{\gamma\circ f^{-1}}(y)\xrightarrow[m,n\to\infty]{as}\gamma\circ f^{-1}(y)$. Combining this with Theorem \ref{t_fm1} completes the proof.
\begin{proof}[Proof of Theorem~\ref{t_f}]
For $F$ an arbitrary function, let $F^{-1}$ denote its generalized inverse, defined by $F^{-1}(t)=\inf\{y:F(y)\geq t\}$. By Theorem \ref{t_fm1}, for all $y\in[f(a),f(b)]$, $$\widehat{{f_{ES}^{-1}}}(y) \xrightarrow[m,n\to\infty]{as} {f_{ES}^{-1}}(y).$$ By Lemma 21.2 in \cite{Vaart98}, $$(\widehat{{f_{ES}^{-1}}})^{-1}(t) \xrightarrow[m,n\to\infty]{as} ({f_{ES}^{-1}})^{-1}(t)$$ at every $t$ where $({f_{ES}^{-1}})^{-1}$ is continuous. Since ${f_{ES}^{-1}}$ is continuous and stricly increasing, $({f_{ES}^{-1}})^{-1}$ is a proper inverse and is equal to $f\circ\phi^{-1}$, hence for all $t\in[a,b]$, $$(\widehat{{f_{ES}^{-1}}})^{-1}(t) \xrightarrow[m,n\to\infty]{as} f\circ\phi^{-1}(t).$$ Now, for all $t\in[a,b]$, $(\widehat{{f_{ES}^{-1}}})^{-1}(t)$ can be rewritten as $$(\widehat{{f_{ES}^{-1}}})^{-1}(t)=v_0 {\bf 1}_{\{a\}}(t)+ \sum_{k=0}^{K(m,n)-1}v_{k+1}{\bf 1}_{(u_k,u_{k+1}]}(t),$$ and by construction, letting $t\in(u_{k(m,n)},u_{k(m,n)+1}]$ $(k(m,n)\in\mathcal K)$, we have $v_{k(m,n)}\leq \widehat{{f_{ES}}}(t) \leq v_{k(m,n)+1}$. Combining this and the above equality yields that for all $t\in (u_{k(m,n)},u_{k(m,n)+1}]$, $|\widehat{{f_{ES}}}(t)-(\widehat{{f_{ES}^{-1}}})^{-1}(t)|\leq v_{k(m,n)+1}-v_{k(m,n)}$. Since ${f_{ES}}$ is continuous, $v_{k(m,n)+1}-v_{k(m,n)}\xrightarrow[m,n\to\infty]{}0$, and $|\widehat{{f_{ES}}}(t)-(\widehat{{f_{ES}^{-1}}})^{-1}(t)|\xrightarrow[m,n\to\infty]{as}0$. It follows that almost surely, $\widehat{{f_{ES}}}(t)$ converges to ${f_{ES}}(t)$. Uniform convergence finally follows from Dini's theorem.
\end{proof}

We now use Theorem \ref{t_h_continu} (given for the continuous model) to prove Theorem \ref{t_h}. A preliminary definition and a lemma are needed.

Let define  the modulus of continuity of $H$ as $$K_H(\delta)=\sup_{\begin{subarray}{c}(s,t) \in[a,b]^2\\ |s-t| \leq \delta \end{subarray}} \left|H_1\circ H_2^{-1}(s)-H_1\circ H_2^{-1}(t)\right|\quad (\delta\geq 0),$$ where $H_1$ and $H_2$ are two independent copies of $H$. Since for all $\delta\geq 0$, $0\leq K_H(\delta)\leq b-a$ almost surely, we can define $L_H^1(H_2, \delta)$ and $L_H^2(H_2, \delta)$ as $L_H^1(H_2, \delta)={\bf E}[K_H(\delta)|H_2]$ and $L_H^2(H_2, \delta)={\bf E}[K_H^2(\delta)|H_2]$ $(\delta\geq 0)$, and we have $L_H^1(H_2,0)=L_H^2(H_2,0)=0$. From continuity of $H$, it holds that $L_H^1(H_2,\cdot)$ and $L_H^2(H_2,\cdot)$ are right-continuous at 0. Given $H_2=h$, this implies that for $\epsilon>0$, there exists $\eta_{\epsilon}>0$ such that $L_H^1(h,\eta_{\epsilon})< \epsilon$. We shall use this result in proving the following lemma:

\begin{lemma}\label{lemtech}
Let $t\in[a,b]$. Under assumption \eqref{d_h}, the following holds:
\begin{equation*}
\sup_{t\in[a,b]}\left|\overline{\phi\circ h_{i_0}^{-1}}\left(t+\frac{1}{n}\right)-\overline{\phi\circ h_{i_0}^{-1}}(t)\right|\xrightarrow[m,n\to\infty]{as}0.
\end{equation*}
\end{lemma}
\begin{proof}
Letting $t\in[a,b]$, one easily shows that
\begin{equation*}
\left|\overline{\phi\circ h_{i_0}^{-1}}\left(t+\frac{1}{n}\right)-\overline{\phi\circ h_{i_0}^{-1}}(t)\right| \leq \frac{1}{m-1} \sum^m_{\begin{subarray}{c}i=1\\i\neq i_0\end{subarray}} \left|h_i\circ h_{i_0}^{-1}\left(t+\frac{1}{n}\right)-h_i\circ h_{i_0}^{-1}(t)\right|.
\end{equation*}
This in turn implies that
\begin{eqnarray*}
\sup_{t\in[a,b]}\left|\overline{\phi\circ h_{i_0}^{-1}}\left(t+\frac{1}{n}\right)-\overline{\phi\circ h_{i_0}^{-1}}(t)\right| &\leq& \frac{1}{m-1} \sum^m_{\begin{subarray}{c}i=1\\i\neq i_0\end{subarray}} \sup_{\begin{subarray}{c}(s,t)\in[a,b]^2\\|s-t|\leq\frac{1}{n}\end{subarray}}\left|h_i\circ h_{i_0}^{-1}(s)-h_i\circ h_{i_0}^{-1}(t)\right|\\
&=&\frac{1}{m-1} \sum^m_{\begin{subarray}{c}i=1\\i\neq i_0\end{subarray}}\tilde K_H^i \left(\frac{1}{n}\right),
\end{eqnarray*}
where the $\tilde K_H^i \left(\frac{1}{n}\right)$ $(i\in\{1,\ldots,m\}\backslash i_0)$ are independent random variables distributed as $K_H\left(\frac{1}{n}\right)|H_2=h_{i_0}$. Since $K_H(\cdot)$ is increasing, if $n\in\mathbb{N}$ is sufficiently large so that $\frac{1}{n}<\eta_{\epsilon}$, we get that
\begin{equation*}
\sup_{t\in[a,b]}\left|\overline{\phi\circ h_{i_0}^{-1}}\left(t+\frac{1}{n}\right)-\overline{\phi\circ h_{i_0}^{-1}}(t)\right| \leq \frac{1}{m-1} \sum^m_{\begin{subarray}{c}i=1\\i\neq i_0\end{subarray}} \tilde K_H^i \left(\eta_{\epsilon} \right).
\end{equation*}
By the law of large numbers, we get that almost surely,
\begin{equation*}
0\leq\lim\sup_{m,n} \sup_{t\in[a,b]}\left|\overline{\phi\circ h_{i_0}^{-1}} \left(t+\frac{1}{n}\right) -\overline{\phi \circ h_{i_0}^{-1}}(t) \right| \leq L_H^1(h_{i_0},\eta_{\epsilon})<\epsilon.
\end{equation*}
This holds for any $\epsilon>0$, hence
\begin{equation*}
\sup_{t\in[a,b]}\left|\overline{\phi\circ h_{i_0}^{-1}}\left(t+\frac{1}{n}\right)-\overline{\phi\circ h_{i_0}^{-1}}(t)\right|\xrightarrow[m,n\to\infty]{as}0.
\end{equation*}
\end{proof}   
\textsc{Proof of Theorem \ref{t_h}} Let $t\in[a,b]$. For each $i\in\{1,\ldots,m\}$, $f_i$ is almost surely increasing, hence almost surely $f_i^{-1}\circ f_{i_0}(t_{j_0})-\frac{1}{n}\leq T_i\leq f_i^{-1}\circ f_{i_0}(t_{j_0})+\frac{1}{n}$. Recall that $t_{j_0}=\arg\min_{j\in{\{1,\dots,n\}}}\left|t_j-t\right|$, hence $t-\frac{1}{n}\leq t_{j_0}\leq t+\frac{1}{n}$. Thus, combining the above two inequalities, we get that almost surely $f_i^{-1}\circ f_{i_0}(t-\frac{1}{n})-\frac{1}{n}\leq T_i\leq f_i^{-1}\circ f_{i_0}(t+\frac{1}{n})+\frac{1}{n}$, from which we deduce:
\begin{equation}\label{e4}
\overline{\phi\circ h_{i_0}^{-1}}\left(t-\frac{1}{n}\right)-\frac{1}{n} \leq \widehat{\phi\circ h_{i_0}^{-1}}(t) \leq\overline{\phi\circ h_{i_0}^{-1}}\left(t+\frac{1}{n}\right)+\frac{1}{n}.
\end{equation}
We shall focus on the upper bound in this inequality, since the same kind of argument holds for the lower bound. Straightforward calculation yields
\begin{eqnarray*}
&&\sup_{t\in[a,b]}\left|\overline{\phi\circ h_{i_0}^{-1}}\left(t+\frac{1}{n}\right)+\frac{1}{n}-\phi\circ h_{i_0}^{-1}(t)\right|\\
&&\hspace{2cm}\leq \sup_{t\in[a,b]}\left|\overline{\phi\circ h_{i_0}^{-1}}\left(t+\frac{1}{n}\right)-\overline{\phi\circ h_{i_0}^{-1}}(t)\right|+\frac{1}{n}+\left\Vert\overline{\phi\circ h_{i_0}^{-1}}-\phi\circ h_{i_0}^{-1}\right\Vert_\infty,
\end{eqnarray*}
where the RHS of this inequality tends to 0 as $m$ and $n$ tend to infinity, by Lemma \ref{lemtech} and Theorem \ref{t_h_continu} in \cite{Billingsley68}. It follows that $\left\Vert\widehat{\phi\circ h_{i_0}^{-1}}-\phi\circ h_{i_0}^{-1}\right\Vert_\infty\xrightarrow[m,n\to\infty]{as}0$.

We now turn to weak convergence. Assume that $n=m^{\frac{1}{2}+\alpha}$ $(\alpha>0)$. From (\ref{e4}), it holds
\begin{eqnarray*}
\left\Vert\sqrt m\left(\widehat{\phi\circ h_{i_0}^{-1}}-\overline{\phi\circ h_{i_0}^{-1}}\right)\right\Vert_\infty&\leq&\sup_{t\in[a,b]}\left|\sqrt m\left(\overline{\phi\circ h_{i_0}^{-1}}\left(t+\frac{1}{n}\right)-\overline{\phi\circ h_{i_0}^{-1}}(t)\right)\right|\\
&&+\sup_{t\in[a,b]}\left|\sqrt m\left(\overline{\phi\circ h_{i_0}^{-1}}\left(t-\frac{1}{n}\right)-\overline{\phi\circ h_{i_0}^{-1}}(t)\right)\right|+\frac{2}{m^{\alpha}}\\
&\leq& 2Z_m+\frac{2}{m^{\alpha}},
\end{eqnarray*}
where $Z_m=\sqrt m\frac{1}{m-1} \sum^m_{\begin{subarray}{c}i=1\\i\neq i_0\end{subarray}} \tilde K_H^i \left(\frac{1}{n}\right)$. Since
\begin{equation*}
\mbox{var}(Z_m)=\frac{m}{m-1}\left(L_H^2\left(h_{i_0},\frac{1}{n}\right)-\left\{L_H^1\left(h_{i_0},\frac{1}{n}\right)\right\}^2\right)\xrightarrow[m\to\infty]{}0
\end{equation*}
(recall that $n=m^{\frac{1}{2}+\alpha}$ and that $L_H^1(h_{i_0},\cdot)$ and $L_H^2(h_{i_0},\cdot)$ are right-continuous at 0), we get that
\begin{equation*}
\left\Vert\sqrt m\left(\widehat{\phi\circ h_{i_0}^{-1}}-\overline{\phi\circ h_{i_0}^{-1}}\right)\right\Vert_\infty \xrightarrow[m\to\infty]{P}0,
\end{equation*}
hence by Theorem \ref{t_h_continu} and Theorem 4.1 in \cite{Billingsley68}, $\sqrt{m}(\widehat{\phi\circ h_{i_0}^{-1}}-\phi\circ h_{i_0}^{-1})\xrightarrow[m\to\infty]{\D}Z$.

\textsc{Proof of Lemma \ref{lcb2}}
Consider first the continuous model (\ref{m_continu}), and define conditionally to $F_{i_0}=f_{i_0}$
\begin{equation*}
\overline{\gamma\circ h_{i_0}^{-1}}=\frac{1}{m-1}\sum_{\begin{subarray}{c}i=1\\i\neq i_0\end{subarray}}^m\left(f_i^{-1}\circ f_{i_0}\right)^2.
\end{equation*}
Using similar arguments as in proof of Theorem \ref{t_h_continu}, we get that $$\left\Vert \overline{\gamma\circ h_{i_0}^{-1}}-\gamma\circ h_{i_0}^{-1}\right\Vert_{\infty}\xrightarrow[m\to\infty]{as}0.$$

Now, from $f_i^{-1}\circ f_{i_0}(t-\frac{1}{n})-\frac{1}{n}\leq T_i\leq f_i^{-1}\circ f_{i_0}(t+\frac{1}{n})+\frac{1}{n}$, we get the following inequality
$$
|T_i-f_i^{-1}\circ f_{i_0}(t)|\leq \left|f_i^{-1}\circ f_{i_0}\left(t-\frac{1}{n}\right)-f_i^{-1}\circ f_{i_0}(t)\right|+\left|f_i^{-1}\circ f_{i_0}\left(t+\frac{1}{n}\right)-f_i^{-1}\circ f_{i_0}(t)\right|+\frac{2}{n}.
$$
Acting as in proof of Lemma \ref{lcb}, it is fairly straightforward to show that for $t\in[a,b]$, $\frac{1}{m-1}\sum_{\begin{subarray}{c}i=1\\i\neq i_0\end{subarray}}^mT_i.f_i^{-1}\circ f_{i_0}(t)\xrightarrow[m,n\to\infty]{as}\gamma\circ h_{i_0}^{-1}(t)$.

From the above inequality, we also obtain that
$$
|T_i-f_i^{-1}\circ f_{i_0}(t)|\leq 2\tilde K_H^i\left(\frac{1}{n}\right)+\frac{2}{n}.
$$
Summing over $i\in\{1,\ldots,m\}\backslash i_0$, and then using convergence of $\frac{1}{m-1}\sum_{i\neq i_0}T_i.f_i^{-1}\circ f_{i_0}(t)$, of $\overline{\gamma\circ h_{i_0}^{-1}}(t)$, together with right-continuity of $L_H^1(h_{i_0},\cdot)$ and $L_H^2(h_{i_0},\cdot)$ at 0, yield that $\widehat{\gamma\circ h_{i_0}^{-1}}(t)\xrightarrow[m,n\to\infty]{as}\gamma\circ h_{i_0}^{-1}(t)$. This and convergence of $\widehat{\phi\circ h_{i_0}^{-1}}(t)$ complete the proof of Lemma \ref{lcb2}.
\bibliographystyle{plain}
\bibliography{traffic}
\end{document}